\documentclass[oneside,english]{amsart}
\usepackage[T1]{fontenc}
\usepackage[latin9]{inputenc}
\usepackage{amstext}
\usepackage{amsthm}
\usepackage{amssymb}
\usepackage{hyperref}

\makeatletter
\numberwithin{equation}{section}
\numberwithin{figure}{section}
\theoremstyle{plain}
\newtheorem{thm}{\protect\theoremname}[section]
\theoremstyle{plain}
\newtheorem{question}[thm]{\protect\questionname}
\theoremstyle{definition}
\newtheorem{example}[thm]{\protect\examplename}
\theoremstyle{definition}
\newtheorem{defn}[thm]{\protect\definitionname}
\theoremstyle{plain}
\newtheorem{lem}[thm]{\protect\lemmaname}
\theoremstyle{plain}
\newtheorem{prop}[thm]{\protect\propositionname}
\theoremstyle{remark}
\newtheorem{rem}[thm]{\protect\remarkname}
\theoremstyle{plain}
\newtheorem{cor}[thm]{\protect\corollaryname}

\usepackage{multicol}
\usepackage{tikz,tikz-cd}
\usetikzlibrary{matrix, arrows}

\makeatletter
\tikzset{
  edge node/.code={%
      \expandafter\def\expandafter\tikz@tonodes\expandafter{\tikz@tonodes #1}}}
\makeatother
\tikzset{
  subseteq/.style={
    draw=none,
    edge node={node [sloped, allow upside down, auto=false]{$\Subset$}}},
  Subseteq/.style={
    draw=none,
    every to/.append style={
      edge node={node [sloped, allow upside down, auto=false]{$\Subset$}}}
  }
}

\newtheorem*{thmaux}{Theorem \theoremauxnum}
\newtheorem*{coraux}{Corollary \theoremauxnum}
\newtheorem*{propaux}{Proposition \theoremauxnum}

\gdef\theoremauxnum{1}

\allowdisplaybreaks

\makeatother

\usepackage{babel}
\providecommand{\corollaryname}{Corollary}
\providecommand{\definitionname}{Definition}
\providecommand{\examplename}{Example}
\providecommand{\lemmaname}{Lemma}
\providecommand{\propositionname}{Proposition}
\providecommand{\questionname}{Question}
\providecommand{\remarkname}{Remark}
\providecommand{\theoremname}{Theorem}

\newcommand{\Miek}{Miek Messerschmidt}

\newcommand{\MiekEmail}{mmesserschmidt@gmail.com}

\newcommand{\UPAddress}{Department of Mathematics and Applied Mathematics; University of Pretoria; Private~bag~X20 Hatfield; 0028 Pretoria; South Africa}

\newcommand{\paperkeywords}{uniform continuity, right inverses, quotient maps, sequence spaces, statistical convergence}

\newcommand{\mscprimary}{%
    46B45, %   	Banach sequence spaces
    40A35, %   	Ideal and statistical convergence
    46B80 %   	Nonlinear classification of Banach spaces; nonlinear quotients
}

\newcommand{\mscsecondary}{
    46B20, %   	Geometry and structure of normed linear spaces
    46B42 %   	Banach lattices
}
\begin{document}
    
    \global\long\def\norm#1{\left\Vert #1\right\Vert }%

\global\long\def\abs#1{\left|#1\right|}%

\global\long\def\set#1#2{\left\{  \vphantom{#1}\vphantom{#2}#1\right.\left|\ #2\vphantom{#1}\vphantom{#2}\right\}  }%

\global\long\def\sphere#1{\mathbf{S}_{#1}}%

\global\long\def\closedball#1{\mathbf{B}_{#1}}%

\global\long\def\openball#1{\mathbb{B}_{#1}}%

\global\long\def\duality#1#2{\left\langle \vphantom{#1}\vphantom{#2}#1\right.\left|\ #2\vphantom{#1}\vphantom{#2}\right\rangle }%

\global\long\def\parenth#1{\left(#1\right)}%

\global\long\def\curly#1{\left\{  #1\right\}  }%

\global\long\def\blockbraces#1{\left[#1\right] }%

\global\long\def\vecspan{\textup{span}}%

\global\long\def\conespan{\textup{cspan}}%

\global\long\def\image{\textup{Im}}%

\global\long\def\support{\textup{supp}}%

\global\long\def\N{\mathbb{N}}%

\global\long\def\I{\mathbb{I}}%

\global\long\def\F{\mathbb{F}}%

\global\long\def\D{\mathbb{D}}%

\global\long\def\R{\mathbb{R}}%

\global\long\def\C{\mathbb{C}}%

\global\long\def\Q{\mathbb{Q}}%

\global\long\def\Rnonneg{\mathbb{R}_{\geq0}}%

\global\long\def\tocorr{\mathbb{\ \twoheadrightarrow\ }}%

\global\long\def\closure#1#2{\overline{#1}^{#2}}%

    \global\long\def\cal#1{\mathcal{#1}}%

\global\long\def\ellp#1{\ell^{#1}}%

\global\long\def\ellinfty{\ellp{\infty}}%

\global\long\def\c{\mathbf{c}}%

\global\long\def\czero{\mathbf{c}_{0}}%

\global\long\def\lip{\textup{Lip}_{0}}%

\global\long\def\FreeLip{F}%

\global\long\def\polar{\odot}%

\global\long\def\conenorm#1{\left\llbracket #1\right\rrbracket }%

\global\long\def\tensor{\otimes}%

\global\long\def\projectivetensor{\overline{\otimes}_{\pi}}%

\global\long\def\isometricisomorphic{\simeq}%

\global\long\def\weak{\textrm{wk}}%

\global\long\def\weakstar{\textrm{wk}^{*}}%

\global\long\def\id{\textrm{id}}%

\global\long\def\ces{\textrm{ces}}%

\global\long\def\Climinf{\textup{Climinf}}%

\global\long\def\Climsup{\textup{Climsup}}%

\global\long\def\ideal{\textup{\textbf{i}}}%

\global\long\def\jdeal{\textup{\textbf{j}}}%

\global\long\def\kdeal{\textup{\textbf{k}}}%

\global\long\def\xdeal#1{\textup{\textbf{#1}}}%

\global\long\def\finideal{\textup{\textbf{f}}}%

\global\long\def\densideal{\textup{\textbf{d}}}%

\global\long\def\banachdensideal{\textup{\textbf{b}}}%

\global\long\def\reciprocalideal{\textup{\textbf{r}}}%

\global\long\def\powerideal{\textup{\textbf{p}}}%

\global\long\def\buckideal{\textup{\textbf{buck}}}%

\global\long\def\buckdensity{\textup{buck}}%

\global\long\def\tileideal{\textup{\textbf{t}}}%

\global\long\def\tiledensity{\textup{tile}}%

\global\long\def\idealczero#1{#1c_{0}}%

\global\long\def\ideallimit#1{\operatorname*{#1\lim}}%

\global\long\def\lcm{\textup{lcm}}%

    \author{\Miek}
    \address{\Miek ;\ \UPAddress}
    \email{\MiekEmail}
    \subjclass[2010]{Primary: \mscprimary; Secondary: \mscsecondary}

    \title[Quotient maps with no uniformly continuous right inverses]{A family of quotient maps of $\ellinfty$ that do not admit uniformly continuous right inverses}
    \begin{abstract}
        Previously only two examples of Banach space quotient maps which do
not admit uniformly continuous right inverses were known: one due
to Aharoni and Lindenstrauss and one due to Kalton ($\ellinfty\to\ellinfty/c_{0}$).

We show through an application of Kalton's Monotone Transfinite Sequence
Theorem that a quotient map of a subspace of $\ellinfty$ of sequences
that converge to zero along an ideal in $\N$ toward another such
subspace, provided one of the ideals is `much larger' than the other,
cannot have a uniformly continuous right inverse. We show in general
that pairs of ideals in $\N$, with one much larger than the other,
occur in abundance.

Some classical examples of ideals in $\N$ presented explicitly are:
the finite subsets of $\N$, the subsets of $\N$ with convergent
reciprocal series, and, the subsets of $\N$ with density zero, Banach
density zero or Buck density zero. 
    \end{abstract}

    \maketitle
    
\section{Introduction}

\medskip

Consider the following question:
\begin{question}
\label{que:Lip-right-inverse-exists}For a given Banach space $X$
and given a closed subspace $E\subseteq X$, does the quotient map
$q:X\to X/E$ admit a uniformly continuous (or even Lipschitz) right
inverse? That is, does there exists a uniformly continuous (or Lipschitz)
map $\tau:X/E\to X$ so that $q\circ\tau=\id|_{X/E}$?
\end{question}

Of course, if $E$ is complemented, then the question is moot. If
uniform continuity is weakened to just continuity, then the answer
is `yes, always' by the Bartle-Graves Theorem \cite[Corollary~17.67]{AliprantisBorder}.

If $E$ is not complemented, then the answer can be either `yes' or
`no' and both alternatives do occur (see \cite{AharoniLindenstrauss}).
Historically, this question is important in studying the nonlinear
geometry of Banach spaces. In \cite{AharoniLindenstrauss}, by constructing
an example of such a quotient map that has a Lipschitz right inverse,
but no bounded linear right inverse, Aharoni and Lindenstrauss constructed
a pair of Banach spaces that are Lipschitz isomorphic, but not linearly
isomorphic, and thereby negatively solving the Lipschitz isomorphism
problem: ``If two Banach spaces are Lipschitz isomorphic, are they
necessarily linearly isomorphic?'' Later, in \cite{KaltonGodefroy2003},
Kalton and Godefroy used the same technique, of constructing a quotient
map that has a Lipschitz right inverse but no continuous linear right
inverse, to prove that each nonseparable weakly compactly generated
Banach space gives rise to a pair of Lipschitz isomorphic Banach spaces
that are not linearly isomorphic \cite[Corollary~4.4]{KaltonGodefroy2003}.
Furthermore, Kalton and Godefroy proved that this method is not available
in the separable case by further proving:
\begin{thm}
\cite[Corollary~3.2]{KaltonGodefroy2003} Let $X$ be a Banach space
and $E\subseteq X$ a closed subspace. If the quotient $X/E$ is separable
and the quotient map $q:X\to X/E$ has a Lipschitz right inverse,
then $q$ also has a continuous linear right inverse, i.e., the subspace
$E$ is complemented.
\end{thm}

For this reason the Lipschitz isomorphism problem for separable Banach
spaces remains an open problem.

Question \ref{que:Lip-right-inverse-exists}, above, is also closely
related to a natural open question on the Lipschitz structure of general
ordered Banach spaces:
\begin{question}
\cite{MesserschmidtBenFeschrift} For a Banach space $X$ ordered
by a closed cone $C\subseteq X$, satisfying $X=C-C$, do there always
exist Lipschitz continuous maps $(\cdot)^{+}:X\to C$ and $(\cdot)^{-}:X\to C$
satisfying $x=x^{+}-x^{-}$ for all $x\in X$?
\end{question}

\bigskip 

To the author's knowledge (and apparently that of the MathOverflow
community \cite{MOquestion}), to date, there are only two examples
of Banach space quotient maps that do not admit uniformly continuous
or Lipschitz right inverses known. They are:
\begin{example}
[Aharoni and Lindenstrauss \cite{AharoniLindenstrauss}] \label{exa:Aharoni-and-Lindenstrauss-quotient-example}Let
$C[0,1]$ denote the space of continuous functions on $[0,1]$ and
let $D_{\Q}[0,1]$ denote the c\'adl\'ag space of all bounded real-valued
functions on $[0,1]$ that are right-continuous everywhere and with
left limits existing everywhere, that only admit discontinuities at
rational numbers. With $D_{\Q}[0,1]$ endowed with the $\norm{\cdot}_{\infty}$-norm,
the quotient map $q:D_{\Q}[0,1]\to D_{\Q}[0,1]/C[0,1]$ does not admit
a uniformly continuous right inverse.
\end{example}

{}
\begin{example}
[Kalton \cite{Kalton2011}] \label{exa:Kalton-quotient-example}The
quotient map $q:\ellinfty\to\ellinfty/c_{0}$ does not admit a uniformly
continuous right inverse.
\end{example}

Kalton's example proceeds through application of Kalton's Monotone
Transfinite Sequence Theorem (\cite[Theorem~4.1]{Kalton2011}, stated
as Theorem~\ref{thm:Kalton's-Monotone-Transfinite-seq-thm} below)
which is also the crucial ingredient in this paper.
\begin{defn}
A Banach lattice $X$ is said to have the \emph{Monotone Transfinite
Sequence Property} if every monotone increasing transfinite sequence
in $X$ (indexed by the first uncountable ordinal) is eventually constant.
\end{defn}

\begin{thm}
[Kalton's Monotone Transfinite Sequence Theorem] \label{thm:Kalton's-Monotone-Transfinite-seq-thm}
If the closed unit ball $\closedball X$ of a Banach lattice $X$
embeds uniformly continuously into $\ellinfty$ through some $f:\closedball X\to\ellinfty$
(with uniformly continuous inverse $g:f(\closedball X)\to X$), then
$X$ has the Monotone Transfinite Sequence Property.
\end{thm}

As mentioned in \cite[Theorem~4.2]{Kalton2011} (cf. Lemma~\ref{lem:strict-f-decreasing-transfinite-sequence-in-N-exist}
below), through a transfinite recursive construction, it is shown
that the quotient $\ellinfty/c_{0}$ does not have the Monotone Transfinite
Sequence Property. From Kalton's Monotone Transfinite Sequence Theorem
(Theorem~\ref{thm:Kalton's-Monotone-Transfinite-seq-thm}) it then
follows easily that the quotient $q:\ellinfty\to\ellinfty/c_{0}$
does not admit a uniformly continuous right inverse.

Only very recently was it proven that subspaces of sequences in $\ellinfty$
that converge to zero with respect to modes of convergence other than
the usual, say in density or Banach density\footnote{Such different modes of convergence, e.g. convergence in density,
make important appearances in other branches of mathematics like ergodic
Ramsey theory (see \cite[Section 4.2]{Furstenberg} for example).}, etc. (cf. Sections~\ref{sec:Preliminaries} and \ref{sec:The-ideals}),
are not complemented in $\ellinfty$ and do not even embed (linearly)
into $\ellinfty$. See Leonetti \cite{Leonetti} and Kania \cite{Kania2019}
respectively. The main aim of this paper is to show, through application
of Kalton's Monotone Transfinite Sequence Theorem, that quotient maps
toward such subspaces of sequences in $\ellinfty$ that converge to
zero with respect to such other modes of convergence, do not admit
uniformly continuous right inverses. We will show that there are an
uncountable number of such subspaces and thereby we provide a large
family of examples of quotient maps of $\ellinfty$ that do not admit
uniformly continuous (or Lipschitz) right inverses, on top of the
previously known two such examples, Examples~\ref{exa:Aharoni-and-Lindenstrauss-quotient-example}
and~\ref{exa:Kalton-quotient-example}, above.

\medskip

In Section~\ref{sec:Preliminaries} we provide the definition of
an ideal $\ideal$ in $\N$, the meaning of $\ideal$-convergence,
and we introduce the notation $\ideal c_{0}$ by which we will denote
the closed subspace of $\ellinfty$ of all elements that $\ideal$-converge
to zero.

In Section~\ref{sec:4 Much larger inclusions of ideals in =00005CN}
we define a strict partial order that we call `much larger inclusion'
(cf. Definition~\ref{def:much-larger}) for general ideals $\ideal$
and $\jdeal$ in $\N$. We use the notation $\ideal\Subset\jdeal.$
Through application of Kalton's Monotone Transfinite Sequence Theorem
(Theorem~\ref{thm:Kalton's-Monotone-Transfinite-seq-thm}), we arrive
at our main tool, Corollary~\ref{cor:jc0 =00005Cto jc0/ic0 has no uniform cont right inverse},
stating that if $\ideal\Subset\jdeal$, then the quotient map $\jdeal c_{0}\to\jdeal c_{0}/\ideal c_{0}$
admits no uniformly continuous right inverse.

We also include some general remarks in Section~\ref{sec:4 Much larger inclusions of ideals in =00005CN}
on the much larger inclusion relation. Firstly, much larger inclusion
is distinct from set inclusion (cf. Remark~\ref{rem:much-larger-inclusion-different-from-inclusion}).
Further, for any ideals $\ideal$ and $\jdeal$ in $\N$ with $\ideal\Subset\jdeal$,
there necessarily exists uncountably many incomparable ideals $\kdeal$
in $\N$, satisfying $\ideal\Subset\kdeal\Subset\jdeal$ (cf. Proposition~\ref{prop:inbetween-lemma}).
Also, there exist maximal $\Subset$-chains (cf. Definition~\ref{def:much-larger-chain})
of ideals in $\N$ with the ideal $\finideal$ of all finite subsets
of $\N$ as smallest element, and $2^{\N}$ as largest element and
such maximal chains are necessarily uncountable as well (cf. Proposition~\ref{prop:maximal-much-larger-chains}).
Therefore, not unexpectedly, much larger inclusions occur in some
abundance within the set of all ideals in $\N$. Since there appears
to be nothing special about any particular instance of a much larger
inclusion, for any ideals $\ideal$ and $\jdeal$ in $\N$ with $\ideal\Subset\jdeal$,
we could conjecture that the quotient space $\jdeal c_{0}/\ideal c_{0}$
might always be isomorphic to $\ellinfty/c_{0}$.

In Section~\ref{sec:The-ideals} we introduce some classical ideals
in $\N$: The ideal $\finideal$ of all finite subsets of $\N$, the
ideal $\reciprocalideal$ of all subsets of $\N$ whose reciprocal
series converge, and the ideals $\densideal,\banachdensideal,\buckideal$
and $\tileideal$ of all subsets of $\N$ whose respective, densities,
Banach densities, Buck densities, and, tile densities are zero. We
show that Buck density and tile density are actually identical, and,
since tile density is easier to handle, we subsequently only consider
tile density.

The final section of the paper is somewhat technical, and is devoted
to establishing the following much larger inclusions of the mentioned
classical ideals in $\N$. They are collected in Corollary~\ref{cor:much-larger-inclusion-relations}
as:

\begin{center}
\begin{tikzcd}  
	\finideal \arrow[Subseteq]{r} &  
	\tileideal\cap\reciprocalideal \arrow[Subseteq]{r}\arrow[Subseteq]{d} & 
	\banachdensideal\cap\reciprocalideal \arrow[Subseteq]{r}\arrow[Subseteq]{d} & 
	\reciprocalideal\arrow[Subseteq]{d} & \\
	&  
	\tileideal \arrow[Subseteq]{r} & 
	\banachdensideal \arrow[Subseteq]{r} & 
	\densideal \arrow[Subseteq]{r} & 
	2^{\N}.
\end{tikzcd}
\end{center}Some of these much larger inclusions are entirely obvious, while a
few others like the much larger inclusions $\banachdensideal\cap\reciprocalideal\Subset\banachdensideal$,
$\tileideal\cap\reciprocalideal\Subset\tileideal$ and $\tileideal\cap\reciprocalideal\Subset\banachdensideal\cap\reciprocalideal$
require some rather careful recursive constructions. Once these much
larger inclusions have been established, by applying Corollary~\ref{cor:jc0 =00005Cto jc0/ic0 has no uniform cont right inverse},
we present a number of explicit quotient maps of $\ellinfty$ and
quotient maps of its subspaces $\densideal c_{0}$, $\banachdensideal c_{0}$,
$\tileideal c_{0}$, $\reciprocalideal c_{0}$, $(\banachdensideal\cap\reciprocalideal)c_{0}$,
$(\tileideal\cap\reciprocalideal)c_{0}$ that do not admit uniformly
continuous right inverses in Corollary~\ref{cor:examples} (but by
Propositions~\ref{prop:inbetween-lemma} and~\ref{prop:maximal-much-larger-chains}
there exist many more, of course).

\section{Preliminaries\label{sec:Preliminaries}}

All vector spaces are assumed to be over $\R$. We assume some minimal
familiarity in the reader with Banach lattices and order ideals in
Banach lattices, the reference \cite{Schaefer} goes far beyond what
we require. 

We define $\N:=\{1,2,3,\ldots\}$ and $\N_{0}:=\N\cup\{0\}$. For
$a,b\in\N_{0}$ with $a\leq b$ we will write $[a,b]:=\{a,a+1,\ldots,b-1,b\}$
and $[a,b):=\{a,a+1,\ldots,b-1\}$. For a set $K\subseteq\N_{0}$
and $a\in\N_{0}$ we define $K+a:=\set{k+a}{k\in K}$ and $aK:=\set{ak}{k\in K}.$
We denote the cardinality of a set $B$ by $\abs B.$
\begin{defn}
A subset $\ideal$ of $2^{\N}$ is called an \emph{ideal in $\N$
}if:
\begin{enumerate}
\item If $A,B\in\ideal$, then $A\cup B\in\ideal$.
\item If $A\in\ideal$ and $B\subseteq A$, then $B\in\ideal$.
\end{enumerate}
\end{defn}

Let $\ideal$ be an ideal in $\N$. For $C,D\subseteq\N$, by $C\subseteq D\mod\ideal$
we will mean that there exists some $I\in\ideal$ so that $C\subseteq D\cup I$
and by $C=D\mod\ideal$ we mean that both $C\subseteq D\mod\ideal$
and $D\subseteq C\mod\ideal$ hold. We will say $C$ and $D$ are
$\ideal$\emph{-disjoint }if $C\cap D\in\ideal$. Let $\kappa$ be
any ordinal and $\{A_{\alpha}\}_{\alpha\in\kappa}$ a (possibly transfinite)
sequence of subsets of $\N$. We will say that $\{A_{\alpha}\}_{\alpha\in\kappa}$
is \emph{$\ideal$-increasing} if $\alpha<\beta<\kappa$ implies $A_{\alpha}\subseteq A_{\beta}\mod\ideal$.
We will say that $\{A_{\alpha}\}_{\alpha\in\kappa}$ is \emph{strictly
$\ideal$-increasing }if $\{A_{\alpha}\}_{\alpha\in\kappa}$ is $\ideal$-increasing
and $\alpha<\beta<\kappa$ implies $A_{\beta}\backslash A_{\alpha}\notin\ideal$.
The terms \emph{$\ideal$-decreasing }and \emph{strict $\ideal$-decreasing}
are defined similarly.

Let $\ideal$ be any ideal in $\N$. We will say that a sequence $\{r_{n}\}_{n\in\N}\subseteq\R$
\emph{$\ideal$-converges to $r\in\R$, }if, for every $\varepsilon>0$,
we have $\{n\in\N:\abs{r_{n}-r}\geq\varepsilon\}\in\ideal$, and we
will write 
\[
\ideallimit{\ideal}_{n\to\infty}r_{n}=r.
\]

As is usual, we denote the space of all bounded sequences in $\R$
by $\ellinfty$, endowed with the $\norm{\cdot}_{\infty}$-norm.
\begin{defn}
Let $\ideal$ be any ideal in $\N$. The closed subspace of $\ellinfty$
of all bounded sequences that $\ideal$-converge to zero will be denoted
$\idealczero{\ideal}:=\set{\xi\in\ellinfty}{\ideallimit{\ideal}_{n\to\infty}\xi_{n}=0}$.
With $\finideal$ denoting the ideal of all finite subsets of $\N$,
as is usual, we will write $c_{0}$ for $\finideal c_{0}$.
\end{defn}

Clearly, for ideals $\ideal$ and $\jdeal$ in $\N$, if $\ideal\subseteq\jdeal$,
then $\idealczero{\ideal}\subseteq\idealczero{\jdeal}$. Further,
since $\ideal c_{0}$ is an order ideal \cite[Definition II.2.1]{Schaefer}
in the Banach lattice $\jdeal c_{0}$, by \cite[Proposition II.5.4]{Schaefer},
the quotient $\jdeal c_{0}/\ideal c_{0}$ is also a Banach lattice.

The following lemma is well known, with an elegant proof from \cite{Whitley},
attributed by Whitley to Kruse.
\begin{lem}
\label{lem:uncountable-pairwise-f-disjoint}There exists an uncountable
collection of pairwise $\finideal$-disjoint infinite subsets of $\N$.
\end{lem}

The following result is mentioned without proof in \cite[Theorem~4.2]{Kalton2011}.
Since our subsequent results rely on it, we provide an explicit proof.
\begin{lem}
\label{lem:strict-f-decreasing-transfinite-sequence-in-N-exist}Let
$\Omega$ denote the first uncountable ordinal. There exists a strictly
$\finideal$-decreasing transfinite sequence of infinite subsets of
$\N$ indexed by $\Omega$. Also, there exists a strictly $\finideal$-increasing
transfinite sequence of infinite subsets of $\N$ indexed by $\Omega$.
\end{lem}

\begin{proof}
Through transfinite recursion, we construct a transfinite sequence
$\{B_{\alpha}\}_{\alpha\in\Omega}$ of strictly $\finideal$-decreasing
infinite subsets of $\N$.

Let $B_{0}:=2\N$.

Let $\kappa\in\Omega$ and assume, that $\{B_{\alpha}\}_{\alpha<\kappa}$
has been defined to be strictly $\finideal$-decreasing. For convenience,
let $\{n_{j}^{(\alpha)}\}_{j\in\N}\subseteq\N$ be a strictly increasing
sequence so that, for all $\alpha<\kappa$, we have $B_{\alpha}=\{n_{j}^{(\alpha)}:j\in\N\}.$

If $\kappa$ is a successor, let $\lambda:=\max\{\alpha|\alpha<\kappa\}$,
so that $\lambda+1=\kappa$. We set 
\[
B_{\kappa}:=\{n_{2j}^{(\lambda)}\in B_{\lambda}:j\in\N\}.
\]
 Clearly $B_{\kappa}$ is infinite. For all $\alpha<\kappa$, we have
$B_{\lambda}\backslash B_{\kappa}\notin\finideal$ and $B_{\lambda}\subseteq B_{\alpha}\mod\finideal$,
and hence $B_{\lambda}\backslash B_{\kappa}\subseteq B_{\alpha}\backslash B_{\kappa}\mod\finideal$,
so that $B_{\alpha}\backslash B_{\kappa}\notin\finideal$.

If $\kappa$ is a limit ordinal, let $(\alpha_{m})_{m\in\N}\subseteq\Omega$
be a strictly increasing sequence of ordinals such that $\alpha_{m}<\kappa$
for all $m\in\N$ and $\sup\{\alpha_{m}:m\in\N\}=\kappa$. Let $p_{1}:=\min B_{\alpha_{1}}\in\N$.
Assuming, for $k\in\N$, that $p_{1}<p_{2}<\ldots<p_{k}$ have been
defined so that $p_{j}\in\bigcap_{r=1}^{j}B_{\alpha_{r}}$ for all
$j\in\{1,\ldots,k\}$, we define $p_{k+1}:=\min\{n\in\bigcap_{r=1}^{k+1}B_{\alpha_{r}}:n>p_{k}\}.$
We obtain a strictly increasing sequence $(p_{j})_{j\in\N}$ so that
$p_{j}\in\bigcap_{r=1}^{j}B_{\alpha_{r}}$ for all $j\in\N$. We define
the infinite set $B_{\kappa}:=\{p_{j}:j\in\N\}.$ For any $\alpha<\kappa\in\Omega$,
let $m\in\N$ satisfy $\alpha<\alpha_{m}<\kappa$, and let $F\in\finideal$
be such that $B_{\alpha_{m}}\subseteq B_{\alpha}\cup F$. Then 
\[
B_{\kappa}\subseteq B_{\alpha_{m}}\cup\{p_{1},\ldots,p_{m-1}\}\subseteq B_{\alpha}\cup(F\cup\{p_{1},\ldots,p_{m-1}\}),
\]
 and hence $B_{\kappa}\subseteq B_{\alpha}\mod\finideal$. Also $B_{\alpha}\backslash B_{\alpha_{m}}\notin\finideal$,
and since $B_{\kappa}\subseteq B_{\alpha_{m}}\mod\finideal$, we have
$B_{\alpha}\backslash B_{\kappa}\notin\finideal$.

We obtain the transfinite sequence $\{B_{\alpha}\}_{\alpha\in\Omega}$
of strictly $\finideal$-decreasing infinite subsets of $\N$.

For $\alpha\in\Omega$, by setting $C_{\alpha}:=\N\backslash A_{\alpha}$,
we obtain a strictly $\finideal$-increasing transfinite sequence
$\{C_{\alpha}\}_{\alpha\in\Omega}$ of infinite sets.
\end{proof}
\begin{prop}
\cite[Theorem~4.2]{Kalton2011} The quotient $\ellinfty/c_{0}$ does
not have the Monotone Transfinite Sequence Property and hence its
closed unit ball $\closedball{\ellinfty/c_{0}}$ cannot embed uniform
continuously into $\ellinfty$. In particular, the quotient map $q:\ellinfty\to\ellinfty/c_{0}$
cannot have a uniformly continuous right inverse.
\end{prop}

\begin{proof}
By Lemma~\ref{lem:strict-f-decreasing-transfinite-sequence-in-N-exist},
there exists a strictly $\finideal$-increasing transfinite sequence
of infinite subsets of $\N$, denoted by $\{B_{\alpha}\}_{\alpha\in\Omega}$.
With $q:\ellinfty\to\ellinfty/c_{0}$ denoting the usual quotient
map and $1_{A}$ the characteristic function of a set $A\in\N$, the
transfinite sequence $\{q(1_{B_{\alpha}})\}_{\alpha\in\Omega}$ is
increasing, but is not eventually constant in the Banach lattice $\ellinfty/c_{0}$.
By Kalton's Monotone Transfinite Sequence Theorem (Theorem~\ref{thm:Kalton's-Monotone-Transfinite-seq-thm}),
the unit ball $\closedball{\ellinfty/c_{0}}$ cannot embed uniformly
continuously into $\ellinfty$. In particular, the quotient map $q:\ellinfty\to\ellinfty/c_{0}$
cannot have a uniformly continuous right inverse.
\end{proof}

\section{Much larger inclusions of ideals in $\protect\N$\label{sec:4 Much larger inclusions of ideals in =00005CN}}

In this section we introduce the strict partial order that we will
term `much larger inclusion' (denoted by `$\Subset$') of ideals of
$\N$. Our main result, Corollary~\ref{cor:jc0 =00005Cto jc0/ic0 has no uniform cont right inverse}
to Lemmas~\ref{lem:muchlarger_implies_transfinite_stirct_i_increasing}
and~\ref{lem:jc0/ic0 does not have MTSP}, is the observation that
if $\ideal\Subset\jdeal$, then the quotient map $\jdeal c_{0}\to\jdeal c_{0}/\ideal c_{0}$
admits no uniformly continuous right inverse.

Before making this observation, we include some general remarks on
much larger inclusions. Firstly, that `much larger inclusion' is different
from set inclusion (cf. Remark~\ref{rem:much-larger-inclusion-different-from-inclusion}).
Secondly, in Propositions~\ref{prop:inbetween-lemma} and~\ref{prop:maximal-much-larger-chains},
through a kind of ``bootstrapping'' of Lemmas~\ref{lem:uncountable-pairwise-f-disjoint}
and~\ref{lem:strict-f-decreasing-transfinite-sequence-in-N-exist},
we show that ideals in $\N$ related by `much larger inclusion' occur
abundance.
\begin{defn}
\label{def:much-larger}Let $\ideal$ and $\jdeal$ be ideals in $\N$.
We will say that \emph{$\ideal$ is much smaller than $\jdeal$} (or
\emph{$\jdeal$ is much larger than $\ideal$}) and write\emph{ $\ideal\Subset\jdeal$,
}if $\ideal\subsetneq\jdeal$ and there exists a countable collection
of non-empty pairwise disjoint sets $\{D_{j}\}_{j\in\N}\subseteq\ideal$
so that, for every infinite subset $K\subseteq\N$, we have that $\bigcup_{k\in K}D_{k}\in\jdeal\backslash\ideal.$
\end{defn}

We note that the relation `$\Subset$' is a strict partial order on
the set of ideals in $\N$.
\begin{rem}
\label{rem:much-larger-inclusion-different-from-inclusion}Strict
inclusion of ideals in $\N$ does not imply much larger inclusion
of ideals in $\N$: Let $\jdeal$ be an ideal in $\N$ that contains
a non-empty set $A\in\jdeal$. Fix some $a\in A$ and define the ideal
$\ideal:=\set{B\in\jdeal}{a\notin B}\subseteq\jdeal$. Since $A\notin\ideal$,
we have $\ideal\subsetneq\jdeal$. By definition of $\ideal$, for
any set $C\in\jdeal\backslash\ideal$, we have $a\in C$. However,
no set in $\ideal$ contains $a$, and therefore no (arbitrary) union
of sets from $\ideal$ can be an element of $\jdeal\backslash\ideal$.
In particular we do not have $\ideal\Subset\jdeal$.
\end{rem}

\begin{prop}
\label{prop:inbetween-lemma}Let $\ideal$ and $\jdeal$ be ideals
in $\N$ with $\ideal\Subset\jdeal$. Then there exist uncountably
many pairwise incomparable ideals strictly between $\ideal$ and $\jdeal$
with respect to the strict partial order $\Subset$.
\end{prop}

\begin{proof}
Let $\{D_{j}\}_{j\in\N}\subseteq\ideal$ be a countable collection
of pairwise disjoint nonempty sets satisfying $\bigcup_{j\in K}D_{j}\in\jdeal\backslash\ideal$
for all infinite sets $K\subseteq\N$. By Lemma~\ref{lem:uncountable-pairwise-f-disjoint},
let $\{L_{i}\}_{i\in I}\subseteq2^{\N}$ be an uncountable collection
of pairwise $\finideal$-disjoint infinite subsets of $\N$.

Let $i\in I$ be arbitrary. Choose strictly increasing disjoint sequences
$\{p_{i,j}\}_{j\in\N}\subseteq\N$ and $\{q_{i,j}\}_{j\in\N}\subseteq\N$
that form a partition of $L_{i}$, i.e., $L_{i}=\{p_{i,j}\}_{j\in\N}\cup\{q_{i,j}\}_{j\in\N}$.
We define the ideal 
\[
\kdeal_{i}:=\set{A\cup B\subseteq\N}{A\in\ideal,\ B\subseteq\bigcup_{j\in\N}D_{p_{i,j}}}.
\]
Consider the countable collection of pairwise disjoint non-empty sets
$\{D_{p_{i,j}}\}_{j\in\N}\subseteq\ideal$. Let $K\subseteq\N$ be
any infinite subset of $\N$. By definition of $\{D_{j}\}_{j\in\N}\subseteq\ideal$
and $\kdeal_{i}$, we have $\bigcup_{j\in K}D_{p_{i,j}}\in\kdeal_{i}\backslash\ideal$.
Therefore $\ideal\Subset\kdeal_{i}$.

On the other hand, consider the countable collection of pairwise disjoint
non-empty sets $\{D_{q_{i,j}}\}_{j\in\N}\subseteq\ideal\subseteq\kdeal_{i}$.
Let $K\subseteq\N$ be any infinite subset of $\N$. We note, by the
fact that $\{D_{j}\}_{j\in\N}$ are pairwise disjoint and that the
sequences $\{p_{i,j}\}_{j\in\N}\subseteq\N$ and $\{q_{i,j}\}_{j\in\N}\subseteq\N$
are disjoint, that the sets $\bigcup_{j\in K}D_{q_{i,j}}\in\jdeal$
and $\bigcup_{j\in\N}D_{p_{i,j}}\in\kdeal_{i}$ are disjoint. So,
if it were the case that $\bigcup_{j\in K}D_{q_{i,j}}\in\kdeal_{i}$,
then, by definition of $\kdeal_{i}$, we must necessarily have $\bigcup_{j\in K}D_{q_{i,j}}\in\ideal$,
which, contradicts $\bigcup_{j\in K}D_{q_{i,j}}\in\jdeal\backslash\ideal$.
Therefore $\bigcup_{j\in K}D_{q_{i,j}}\in\jdeal\backslash\kdeal_{i}$.
Since $K\subseteq\N$ was chosen arbitrarily as an infinite set, we
obtain $\kdeal_{i}\Subset\jdeal$.

Now let $i_{0},i_{1}\in I$ be distinct. We claim that $\kdeal_{i_{0}}$
and $\kdeal_{i_{1}}$ are not comparable with respect to the strict
partial order $\Subset$. Since $\{L_{i}\}_{i\in I}\subseteq2^{\N}$
is a collection of pairwise $\finideal$-disjoint sets, $\{p_{i_{0},j}\}_{j\in\N}\cap\{p_{i_{1},j}\}_{j\in\N}\subseteq L_{i_{0}}\cap L_{i_{1}}\in\finideal$.
Therefore there exists some $J\in\N$, such that $\{p_{i_{0},j}\}_{j\geq J}$
and $\{p_{i_{1},j}\}_{j\geq J}$ are disjoint, and, since the elements
of $\{D_{j}\}_{j\in\N}\subseteq\ideal$ are pairwise disjoint, the
sets $\bigcup_{j=J}^{\infty}D_{p_{i_{0},j}}\in\kdeal_{i_{0}}$ and
$\bigcup_{j=J}^{\infty}D_{p_{i_{1},j}}\in\kdeal_{i_{1}}$ are thus
disjoint. Therefore, if it were the case that $\bigcup_{j=J}^{\infty}D_{p_{i_{0},j}}\in\kdeal_{i_{1}}$,
then there necessarily exists some $A\in\ideal$ such that $\bigcup_{j=J}^{\infty}D_{p_{i_{0},j}}\subseteq A\cup\bigcup_{j=1}^{J-1}D_{p_{i_{1},j}}$.
Since the elements of $\{D_{j}\}_{j\in\N}\subseteq\ideal$ are pairwise
disjoint, there must exist an infinite set $K\subseteq\N$ so that
$\bigcup_{j\in K}D_{p_{i_{0},j}}\subseteq A\in\ideal$. But this contradicts
$\bigcup_{j\in K}D_{p_{i_{0},j}}\in\jdeal\backslash\ideal$. Therefore
$\bigcup_{j=J}^{\infty}D_{p_{i_{0},j}}\in\kdeal_{i_{0}}$, while $\bigcup_{j=J}^{\infty}D_{p_{i_{0},j}}\notin\kdeal_{i_{1}}$.
Similarly, $\bigcup_{j=J}^{\infty}D_{p_{i_{1},j}}\in\kdeal_{i_{1}}$,
while $\bigcup_{j=J}^{\infty}D_{p_{i_{1},j}}\notin\kdeal_{i_{0}}$.

The elements of the uncountable collection $\{\kdeal_{i}\}_{i\in I}$
are thus pairwise incomparable, and lie between $\ideal$ and $\jdeal$
with respect to the strict partial order $\Subset$.
\end{proof}
\begin{defn}
\label{def:much-larger-chain}By a $\Subset$-chain, we mean a set
of ideals in $\N$ which is strictly totally ordered by the strict
partial order $\Subset$.
\end{defn}

\begin{prop}
\label{prop:maximal-much-larger-chains}Let $\ideal$ and $\jdeal$
be ideals in $\N$, with $\ideal\Subset\jdeal$. There exists a maximal
$\Subset$-chain having $\ideal$ as the least element and $\jdeal$
as greatest element and any such a maximal $\Subset$-chain is uncountable.
In particular, there exists a maximal $\Subset$-chain having $\finideal$
as the least element and $2^{\N}$ as greatest element and any such
maximal $\Subset$-chain is uncountable.
\end{prop}

\begin{proof}
That a maximal $\Subset$-chain exists between $\ideal$ and $\jdeal$
is a straightforward application of Zorn's lemma.

Let $\{D_{j}\}_{j\in\N}\subseteq\ideal$ be a countable collection
of nonempty pairwise disjoint subset so that for any infinite set
$K\subseteq\N$ we have $\bigcup_{j\in K}D_{j}\in\jdeal\backslash\ideal$.
By Lemma~\ref{lem:strict-f-decreasing-transfinite-sequence-in-N-exist}
there exists a transfinite strictly $\finideal$-increasing sequence
$\{C_{\alpha}\}_{\alpha\in\Omega}$ of infinite subsets of $\N$.
For each $\alpha\in\Omega$, define the ideal 
\[
\kdeal_{\alpha}:=\set{A\cup B\subseteq\N}{A\in\ideal,\ B\subseteq\bigcup_{j\in C_{\alpha}}D_{j}}.
\]
Let $\beta<\gamma<\Omega$. Since $C_{\gamma}\backslash C_{\beta}$
is infinite, there exists a strictly increasing sequence $\{p_{j}\}_{j\in\N}\subseteq C_{\gamma}\backslash C_{\beta}.$
We consider the countable collection of nonempty pairwise disjoint
sets $\{D_{p_{j}}\}_{j\in\N}\subseteq\ideal\subseteq\kdeal_{\beta}$.
Let $K\subseteq\N$ be any infinite set. It is clear that $\bigcup_{j\in K}D_{p_{j}}\subseteq\bigcup_{j\in C_{\gamma}}D_{j}\in\kdeal_{\gamma}$,
while $\bigcup_{j\in K}D_{p_{j}}$ is disjoint from $\bigcup_{j\in C_{\beta}}D_{j}$.
Therefore, if it were the case that $\bigcup_{j\in K}D_{p_{j}}\in\kdeal_{\beta}$
then, it is necessarily true that $\bigcup_{j\in K}D_{p_{j}}\in\ideal$,
which is false. Therefore $\kdeal_{\beta}\Subset\kdeal_{\gamma}$.
Since $\Omega$ is uncountable, we conclude that any maximal $\Subset$-chain
of ideals must be uncountable.
\end{proof}
\begin{lem}
\label{lem:muchlarger_implies_transfinite_stirct_i_increasing}Let
$\ideal$ and $\jdeal$ be ideals in $\N$, with $\ideal\Subset\jdeal$.
There exists a strictly $\ideal$-increasing transfinite sequence
of infinite sets $\{B_{\alpha}\}_{\alpha\in\Omega}\subseteq\jdeal\backslash\ideal$.
\end{lem}

\begin{proof}
By Lemma~\ref{lem:strict-f-decreasing-transfinite-sequence-in-N-exist}
let $\{C_{\alpha}\}_{\alpha\in\Omega}$ be a strictly $\finideal$-increasing
transfinite sequence of infinite subsets of $\N$. Since $\jdeal$
is much larger than $\ideal$, there exists a collection of non-empty
pairwise disjoint sets $\{D_{j}\}_{j\in\N}\subseteq\ideal$ so that
for every infinite subset $K\subseteq\N$ we have that $\bigcup_{k\in K}D_{k}\in\jdeal\backslash\ideal.$
For $\alpha\in\Omega$, we define $B_{\alpha}:=\bigcup_{j\in C_{\alpha}}D_{j}$.
Clearly, for $\alpha\in\Omega$, since $C_{\alpha}$ is infinite,
$B_{\alpha}$ is an infinite set.

Let $\alpha<\beta\in\Omega$ be arbitrary. Since the elements of $\{D_{j}\}_{j\in\N}$
are pairwise disjoint, we have that $B_{\beta}\backslash B_{\alpha}=\bigcup_{j\in C_{\beta}\backslash C_{\alpha}}D_{j}$.
But $C_{\beta}\backslash C_{\alpha}$ is infinite, and therefore $B_{\beta}\backslash B_{\alpha}=\bigcup_{j\in C_{\beta}\backslash C_{\alpha}}D_{j}\notin\ideal.$
Also $C_{\alpha}\subseteq C_{\beta}\mod\finideal$. Let $F\in\finideal$
be such that $C_{\alpha}\subseteq C_{\beta}\cup F$ so that $B_{\alpha}=\bigcup_{j\in C_{\alpha}}D_{j}\subseteq\parenth{\bigcup_{j\in C_{\beta}}D_{j}}\cup\parenth{\bigcup_{j\in F}D_{j}}$.
But the finite union $\bigcup_{j\in F}D_{j}$ is an element of $\ideal$,
and hence $B_{\alpha}\subseteq B_{\beta}\mod\ideal$. Since $\alpha<\beta\in\Omega$
were chosen arbitrarily, $\{B_{\alpha}\}_{\alpha\in\Omega}\subseteq\jdeal$
is strictly $\ideal$-increasing.
\end{proof}
{}
\begin{lem}
\label{lem:jc0/ic0 does not have MTSP}Let $\ideal$ and $\jdeal$
be ideals in $\N$, with $\ideal\Subset\jdeal$. The quotient $\idealczero{\jdeal}/\idealczero{\ideal}$
does not have the Monotone Transfinite Sequence Property.
\end{lem}

\begin{proof}
Let $q:\idealczero{\jdeal}\to\idealczero{\jdeal}/\idealczero{\ideal}$
denote the quotient map and, for any $A\subseteq\N$, let $1_{A}$
denote the characteristic function of $A$. By Lemma~\ref{lem:muchlarger_implies_transfinite_stirct_i_increasing},
there exists a strictly $\ideal$-increasing transfinite sequence
$\{B_{\alpha}\}_{\alpha\in\Omega}\subseteq\jdeal\backslash\ideal$
of infinite subsets of $\N$. The transfinite sequence $\{q(1_{B_{\alpha}})\}_{\alpha\in\Omega}\subseteq\idealczero{\jdeal}/\idealczero{\ideal}$
is strictly increasing in the Banach lattice $\idealczero{\jdeal}/\idealczero{\ideal}$
and not eventually constant.
\end{proof}
\begin{cor}
\label{cor:jc0 =00005Cto jc0/ic0 has no uniform cont right inverse}Let
$\ideal$ and $\jdeal$ be ideals in $\N$, with $\ideal\Subset\jdeal$.
The unit ball $\closedball{\idealczero{\jdeal}/\idealczero{\ideal}}$
does not embed uniform continuously into $\idealczero{\jdeal}$. In
particular, $\idealczero{\ideal}$ is not complemented in $\idealczero{\jdeal}$,
and furthermore, the quotient map $q:\idealczero{\jdeal}\to\idealczero{\jdeal}/\idealczero{\ideal}$
does not admit a uniform continuous right inverse.
\end{cor}

\begin{proof}
Since $\idealczero{\jdeal}/\idealczero{\ideal}$ does not have the
Monotone Transfinite Sequence Property, by Kalton's Monotone Transfinite
Sequence Theorem (Theorem~\ref{thm:Kalton's-Monotone-Transfinite-seq-thm}),
the closed unit ball $\closedball{\idealczero{\jdeal}/\idealczero{\ideal}}$
cannot embed uniform continuously into $\idealczero{\jdeal}$ as $\closedball{\idealczero{\jdeal}/\idealczero{\ideal}}$
would then embed uniformly continuously into $\ellinfty$. In particular,
the quotient map $q:\idealczero{\jdeal}\to\idealczero{\jdeal}/\idealczero{\ideal}$
cannot have a uniformly continuous right inverse.
\end{proof}

\section{Some explicit ideals in $\protect\N$: $\protect\finideal$, $\protect\reciprocalideal$,
$\protect\densideal$, $\protect\banachdensideal$, $\protect\buckideal$,
and $\protect\tileideal$.\label{sec:The-ideals}}

In this section we explicitly define a number of classical ideals
in $\N$. In the next section we establish how these ideals relate
to each other through the strict partial order $\Subset$. Most of
the ideals in $\N$ that we consider are defined by using some notion
of density of sets in $\N$. We refer the reader to \cite{leonetti_tringali}
for further context. 

\medskip

By $\finideal$ we will denote the ideal of all finite subsets of
$\N$.

For a set $X\subseteq\N$, we will call the series (or sum if $X$
is finite) $\sum_{x\in X}x^{-1}$ the \emph{reciprocal series of $X$.
}By $\reciprocalideal:=\set{X\subseteq\N}{\sum_{x\in X}x^{-1}<\infty}$
we denote the ideal of all subsets of $\N$ whose reciprocal series
converge.

We define the \emph{density }of a set $X\subseteq\N$ by 
\[
d(X):=\lim_{n\to\infty}\frac{\abs{X\cap[1,n]}}{n},
\]
(if this limit exists) and the ideal $\densideal:=\{X\subseteq\N:d(X)=0\}$.

We define the\emph{ Banach density} of a set $X\subseteq\N$ by
\[
b(X):=\lim_{n\to\infty}\max_{k\in\N_{0}}\frac{\abs{X\cap[k+1,k+n]}}{n},
\]
and the ideal $\banachdensideal:=\{X\subseteq\N:b(X)=0\}.$

Define the set of all finite unions of infinite arithmetic progressions
as
\[
\cal R:=\set{\bigcup_{i=1}^{n}(a_{i}\N_{0}+b_{i})}{n\in\N,\ \{a_{i}\}_{i=1}^{n},\{b_{i}\}_{i=1}^{n}\subseteq\N}.
\]
We define the \emph{Buck density} of a set $X\subseteq\N$ by
\[
\buckdensity(X):=\inf\set{d(B)}{B\in\cal R,\ X\subseteq B\mod\finideal},
\]
and the ideal $\buckideal:=\{X\subseteq\N:\buckdensity(X)=0\}$.

For $r\in\N$, we define $\cal T_{r}:=2^{[0,r-1]}$ and $\cal{PT}_{r}:=2^{[0,r-1]}\backslash\{[0,r-1]\}$.
An element of $\cal T_{r}$ will be called a \emph{tile (of length
$r$)}. An element of $\cal{PT}_{r}$ will be called a \emph{proper
tile (of length $r$).} Define the set of all \emph{infinite tilings}
as
\[
\cal T:=\set{\parenth{\bigcup_{j=0}^{\infty}T+jr}+b}{r,b\in\N,\ T\in\cal T_{r}}\subseteq2^{\N}.
\]
We define the \emph{tile density} of a set $X\subseteq\N$ by
\[
\tiledensity(X):=\inf\set{d(A)}{A\in\cal T,\ X\subseteq A\mod\finideal},
\]
and the ideal $\tileideal:=\{X\subseteq\N:\tiledensity(X)=0\}$.

We remark in the following that Buck density and tile density are
actually the same.
\begin{lem}
\label{lem:A-and-T-are the same-mod-f}The inclusion $\cal T\subseteq\cal R$
holds and for every $B\in\cal R$ there exists some $C\in\cal T$
so that $B=C\mod\finideal$.
\end{lem}

\begin{proof}
It is clear that $\cal T\subseteq\cal R$.

Let $B\in\cal R$ with $n\in\N,\ \{a_{i}\}_{i=1}^{n},\{b_{i}\}_{i=1}^{n}\subseteq\N$,
so that $B=\bigcup_{i=1}^{n}(a_{i}\N_{0}+b_{i})$. Let $r:=\lcm\{a_{i}\}_{i=1}^{n}$
and $b:=\max\{b_{i}\}_{i=1}^{n}$. Define $T:=B\cap[b,b+r)-b$ and
$C:=\parenth{\bigcup_{j=0}^{\infty}T+jr}+b$. Then $B\subseteq C\mod\finideal$
and $C\subseteq B\mod\finideal$.
\end{proof}
\begin{cor}
The Buck density and the tile density are identical.
\end{cor}

\begin{proof}
Let $X\subseteq\N$ be any set. For every $B\in\cal R$ with $X\subseteq B\mod\finideal$,
by Lemma~\ref{lem:A-and-T-are the same-mod-f}, there exists some
$C\in\cal T$ with $B=C\mod\finideal$. Since $B$ and $C$ differ
by at most a finite set, we have $d(B)=d(C)$. Therefore $\tiledensity(X)\leq\buckdensity(X)$.
On the other hand, since $\cal{T\subseteq\cal R}$ we have $\buckdensity(X)\leq\tiledensity(X)$,
so that $\buckdensity(X)=\tiledensity(X)$.
\end{proof}
In light of the previous corollary and that, for our purposes, the
tile density is easier to handle, we will work with the tile density
instead of the Buck density.

\section{Much larger inclusions of some explicit ideals}

In this section we will prove Corollary~\ref{cor:much-larger-inclusion-relations}
which displays the the much larger inclusions of the ideals defined
in the previous section by means of a number of lemmas. The arguments
are entirely elementary, but do become increasingly technical in this
section. The reader will be sometimes be primed with an intuitive
explanation of the arguments involved in subsequent results, the author's
hope being to aid in understanding the technical aspects more easily. 

The first two lemmas are entirely straightforward.
\begin{lem}
\label{lem:much-larger-A}The much larger inclusion $\finideal\Subset\tileideal\cap\reciprocalideal$
holds.
\end{lem}

\begin{proof}
For every $j\in\N$ we define $B_{j}:=\{2^{j}\}$. Let $K\subseteq\N$
be any infinite set. It is clear that $\bigcup_{k\in K}B_{k}\notin\finideal$.
Since $\sum_{k=1}^{\infty}2^{-k}$ converges, it is also clear that
$\bigcup_{k\in K}B_{k}\in\reciprocalideal$. For every $n\in\N$,
let $T:=\{2^{n}-1\}\in\cal T_{2^{n}}$. It is further clear that $\bigcup_{k\in K}B_{k}\subseteq\bigcup_{j=0}^{\infty}\parenth{T+2^{n}j}+1\mod\finideal$.
While $d\parenth{\bigcup_{j=0}^{\infty}\parenth{T+2^{n}j}+1}=2^{-n}.$
Therefore $\tiledensity\parenth{\bigcup_{k\in K}B_{k}}=0$, so that
$\bigcup_{k\in K}B_{k}\in\tileideal$, and hence $\bigcup_{k\in K}B_{k}\in\tileideal\cap\reciprocalideal$.
\end{proof}
\begin{lem}
\label{lem:much-larger-B}The much larger inclusions $\banachdensideal\cap\reciprocalideal\Subset\reciprocalideal$
and $\banachdensideal\Subset\densideal$ hold.
\end{lem}

\begin{proof}
For each $j\in\N$, define $B_{j}:=[2^{j}+1,2^{j}+j]\in\banachdensideal\cap\reciprocalideal.$
Let $K\subseteq\N$ be any infinite set. Then
\[
\sum_{b\in\bigcup_{j\in K}B_{j}}b^{-1}\leq\sum_{b\in\bigcup_{j\in\N}B_{j}}b^{-1}<\sum_{j=1}^{\infty}j2^{-j}<\infty,
\]
and hence $\bigcup_{j\in K}B_{j}\in\reciprocalideal$, while $b(\bigcup_{j\in K}B_{j})=1$
is clear from the definition of Banach density. Therefore $\bigcup_{j\in K}B_{j}\notin\banachdensideal$.

We claim that $\bigcup_{j\in K}B_{j}\in\densideal$. With $n\in\N$,
let $m=\min\{j\in\N:n\leq2^{j}\}$. Then
\begin{align*}
\frac{\abs{\bigcup_{j\in K}B_{j}\cap[1,n]}}{n} & \leq\frac{\abs{\bigcup_{j\in\N}B_{j}\cap[1,n]}}{n}\\
 & \leq\frac{\abs{\bigcup_{j\in\N}B_{j}\cap[1,2^{m}]}}{2^{m-1}}\\
 & <\frac{1+2+3+\ldots+m}{2^{m-1}}\\
 & =\frac{m(m+1)}{2^{m}}.
\end{align*}
Since $m\to\infty$ as $n\to\infty$, we have $d(\bigcup_{j\in K}B_{j})=0,$
and hence $\bigcup_{j\in K}B_{j}\in\densideal$.
\end{proof}
{}

In the next lemma we establish the much larger inclusions $\banachdensideal\cap\reciprocalideal\Subset\banachdensideal$
and $\reciprocalideal\Subset\densideal$. The idea is to consider
the nested infinite arithmetic progressions $2^{j}\N$ for $j\in\N$.
Let $j\in\N$. With $a,b\in\N$, if $b-a>0$ is large enough, the
set $2^{j}\N\cap[a,b]$ can be made to have arbitrarily large reciprocal
series. On the other hand, for any finite set $F\subseteq\N$, the
set $(F\cup2^{j}\N)$ has Banach density $2^{-j}$. This in mind,
one proceeds in choosing disjoint finite sets from the arithmetic
progressions $2^{j}\N\ (j\in\N)$ with increasingly large reciprocal
sums while having decreasingly small Banach density.
\begin{lem}
\label{lem:much-larger-D}The much larger inclusions $\banachdensideal\cap\reciprocalideal\Subset\banachdensideal$
and $\reciprocalideal\Subset\densideal$ hold.
\end{lem}

\begin{proof}
For all $j\in\N$, define $L_{j}:=2^{j}\N$ and let $B_{1}:=\{2\}$.
We recursively construct a collection $\{B_{j}\}_{j\in\N}$ of disjoint
finite subsets of $\N$ as follows. Let $k\in\N$, and assume that,
for all $1\leq j\leq k$, that $\{B_{j}\}_{j=1}^{k}$ has been defined
as a collection of finite pairwise disjoint subsets of $\N$, so that
$B_{j}\subseteq L_{j}$, the reciprocal series of $B_{j}$ is greater
or equal to $\sum_{m=2}^{j+1}m^{-1}$, and that there exists some
$p\in\N$ so that for all $q\geq p$ and $v\in\N_{0}$ we have.

\[
\frac{\abs{\parenth{\bigcup_{j=1}^{k}B_{j}\cup L_{k+1}}\cap[v+1,v+q]}}{q}<\frac{1}{k}.
\]
We choose $s\in\N$, and $p'\in\N$ so that the following two conditions
hold: Firstly, the set $B_{k+1}:=[\max B_{k}+1,s]\cap L_{k+1}$ is
required to have a reciprocal series greater than $\sum_{m=2}^{(k+1)+1}m^{-1}$.
This is possible, since $L_{k+1}$ is an infinite arithmetic progression.
Secondly, for all $q\geq p'$ and $v\in\N_{0}$ we require that 
\[
\frac{\abs{\parenth{\bigcup_{j=1}^{k+1}B_{j}\cup L_{k+1}}\cap[v+1,v+q]}}{q}<\frac{1}{k+1}.
\]
This is possible since $L_{k+1}$ has Banach density $2^{-(k+1)}$.

Let $K\subseteq\N$ be infinite. With $\{B_{j}\}_{j\in\N}$ as constructed,
we then have that the reciprocal series of $\bigcup_{j\in K}B_{j}$
diverges, and hence that $\bigcup_{j\in K}B_{j}\notin\reciprocalideal$.
Also for every $k\in K$ we have $\bigcup_{j\in K}B_{j}\subseteq\bigcup_{j=1}^{k}B_{j}\cup L_{k}$,
and there exists some $p\in\N$ so that for all $q\geq p$
\[
\frac{\abs{\parenth{\bigcup_{j=1}^{k}B_{j}\cup L_{k}}\cap[v+1,v+q]}}{q}<\frac{1}{k}.
\]
Since $K$ is infinite, we therefore have $\bigcup_{j\in K}B_{j}\in\banachdensideal\subseteq\densideal$.
\end{proof}
In the next lemma we show that $\tileideal\cap\reciprocalideal\Subset\tileideal$.
The main intuition is a recursive application of the following: We
start with some non-empty finite set $T\subseteq\N_{0}$, a tile,
with its length $r(>\max T)$. For any fixed $a\in\N$, by choosing
$n\in\N$ and $b\in\N$ large enough, the reciprocal series of $[a,b]\cap\bigcup_{j=0}^{n}(T+jr+1)$
can be made arbitrarily large. For $j\in\{1,\ldots,n\}$ setting $S_{j}:=T$
and for $j>n$ setting $S_{j}:=\emptyset$, if $m\in\N$ is large
enough, then 
\[
\frac{\abs{[1,(m+1)r]\cap\bigcup_{j=0}^{m}(S_{j}+jr+1)}}{(m+1)r-1}
\]
can be made arbitrarily small. This process is then recursively repeated
(with the tile $\bigcup_{j=0}^{m}(S_{j}+jr)$ of length $r':=(m+1)r$
in the next step) so as to construct a sequence of tiles whose infinite
tilings have decreasing density. One proceeds to choose a sequence
of pairwise disjoint finite sets from these nested infinite tilings
with increasing reciprocal series.
\begin{lem}
\label{lem:much-larger-C}The much larger inclusion $\tileideal\cap\reciprocalideal\Subset\tileideal$
holds.
\end{lem}

\begin{proof}
We define $T_{1}:=\{0\}$, $r_{0}:=0$ and $r_{1}:=2$.

Let $k\in\N$. We assume that sequence of natural numbers $r_{1}<r_{2}<r_{3}<\ldots<r_{k}$
with $r_{j+1}$ a multiple of $r_{j}$ for $j\in\{1,\ldots,k-1\}$
have been and sets $\{0\}=T_{1}\subseteq T_{2}\subseteq\ldots\subseteq T_{k}\subseteq\N_{0}$
have been defined to satisfy the following:
\begin{enumerate}
\item For all $i\in\{1,\ldots,k\}$, we have $T_{i}\in\cal{PT}_{r_{i}}$.
\item For all $i\in\{1,\ldots,k-1\}$, we have $T_{i+1}\subseteq\bigcup_{j=0}^{\infty}(T_{i}+jr_{i})$.
\item For all $i\in\{1,\ldots,k\},$ $\sum_{t\in T_{i}}(t+r_{i-1}+1)^{-1}\geq i$.
\item For all $i\in\{1,\ldots,k\}$, we have $\abs{\bigcup_{j=1}^{i}(T_{j}+r_{j-1}+1)}/\parenth{\sum_{j=1}^{i}r_{j}}<i^{-1}.$
\item For all $i\in\{1,\ldots,k\}$, we have $\abs{T_{i}}/r_{i}<i^{-1}.$
\end{enumerate}
Let $s\in\N$ be the least number so that $\bigcup_{j=0}^{s}(T_{k}+(r_{k}+1)+jr_{k})$
has reciprocal series greater or equal to $k$. Let $T_{k+1}:=\bigcup_{j=0}^{s}(T_{k}+jr_{k})$.
Let $r_{k+1}\in\N$ be the least multiple of $r_{k}$ so that $\abs{\bigcup_{j=1}^{k+1}(T_{j}+r_{j-1}+1)}/\sum_{j=1}^{k+1}r_{k+1}<(k+1)^{-1}$
and $\abs{T_{k+1}}/r_{k+1}<(k+1)^{-1}$.

With $\{T_{i}\}_{i=1}^{\infty}$ and $\{r_{i}\}_{i=1}^{\infty}$ recursively
defined as above, for every $j\in\N$, we set $B_{j}:=T_{j}+(r_{j-1}+1)$.
Then $\{B_{j}\}_{j=1}^{\infty}\subseteq\finideal\subseteq\tileideal\cap\reciprocalideal$
is a collection non-empty finite pairwise disjoint sets.

For any infinite set $K\subseteq\N$, the set $\bigcup_{k\in K}B_{k}$
has a divergent reciprocal series so $\bigcup_{k\in K}B_{k}\notin\reciprocalideal$.
This, while for every $i\in\N$, we have $\bigcup_{k\in K}B_{k}\subseteq\bigcup_{j=0}^{\infty}(T_{i}+jr_{i}+1)$.
But 
\[
d\parenth{\bigcup_{j=0}^{\infty}(T_{i}+jr_{i}+1)}=\abs{T_{i}}/r_{i}<i^{-1},
\]
and therefore $\tiledensity\parenth{\bigcup_{k\in K}B_{k}}=0$, so
that $\bigcup_{k\in K}B_{k}\in\tileideal$.
\end{proof}
In our final lemma we will show that $\tileideal\cap\reciprocalideal\Subset\banachdensideal\cap\reciprocalideal$
and $\tileideal\Subset\banachdensideal$. Before embarking on this
verification, we define what we term an \emph{anti-tile}.
\begin{defn}
Let $r,s\in\N$ with $r\leq s$ and choose a labeling of the elements
of $\cal{PT}_{r}$ as $\{T_{i}\}_{i=0}^{(2^{r}-1)-1}$. We define
an \emph{anti-tile (of order $r$, step-length $s$, and length $(2^{r}-1)s$)
}as any set of the form 
\[
A(r,s):=\bigcup_{i=0}^{(2^{r}-1)-1}\{\min(\N_{0}\backslash T_{i})\}+is.
\]
\end{defn}

In aid of the exposition that follows, we first display the following
example:
\begin{example}
\label{exa:A(3,3)}For a set $B\subseteq\N_{0}$ and all $j\in\N_{0}$,
define $b_{j}:=1$ if $j\in B$ and otherwise $b_{j}:=0$, and let
$\text{bin}(B):=b_{0}b_{1}b_{2}\ldots$, truncating a tail of all-zeros
anywhere if $B$ is finite. Let $\{T_{i}\}_{i=0}^{7-1}$ be a labelling
of $\cal{PT}_{3}$ according to increasing binary representation.
With commas inserted for readability, we have
\begin{align*}
\text{bin}(A(3,3)) & =\texttt{100,100,100,100,010,010,001}\\
\text{bin}(\bigcup_{i=0}^{7-1}T_{i}+3i) & =\texttt{000,001,010,011,100,101,110}.
\end{align*}
By construction, for every $j\in\{0,\ldots,6\}$, there exists a position
where a $1$ occurs in the string $\text{bin}\parenth{A(3,3)}$ and
where a $0$ occurs at the corresponding position in the string $\text{bin}\parenth{\bigcup_{i=0}^{6}T_{j}+3i}$,
e.g., 
\begin{align*}
\text{bin}\parenth{A(3,3)} & =\texttt{100100100100010010001}\\
\text{bin}\parenth{\bigcup_{i=0}^{7-1}T_{1}+3i} & =\texttt{100100100100100100100}\\
\text{bin}\parenth{\bigcup_{i=0}^{7-1}T_{5}+3i} & =\texttt{101101101101101101101}\\
\text{bin}\parenth{\bigcup_{i=0}^{7-1}T_{6}+3i} & =\texttt{110110110110110110110}.
\end{align*}
This feature extends to translates of infinite tilings: For any $a,b\in\N_{0}$
and for every $j\in\{0,\ldots,6\}$, there exists infinitely many
positions where a $1$ occurs in the string $\text{bin}\parenth{\bigcup_{i=0}^{\infty}A(3,3)+(7\times3)j+a}$,
and where, in the corresponding position, a $0$ occurs in the string
$\text{bin}\parenth{\bigcup_{i=0}^{\infty}T_{j}+3j+b}$.\bigskip 
\end{example}

For $r\in\N$, an anti-tile $A(r',s)$ of order $r'\geq r$ and step
length $s\geq r'$ with $s\in r\N$, is with so named, because $A(r',s)$
is constructed exactly to be such that an infinite tiling of $A(r',s)$
is not contained mod $\finideal$ in any infinite tiling of any tile
from $\cal{PT}_{r}$. Explicitly, let $r,r'\in\N$, with $r'\geq r$.
Let $s\in r\N$ with $s\geq r'$. For any $T\in\cal{PT}_{r}$ and
$b\in\N$ we have, 
\[
\abs{\parenth{\bigcup_{j=0}^{\infty}A(r',s)+j(2^{r'}-1)s+1}\backslash\parenth{\bigcup_{j=0}^{\infty}(T+jr)+b}}=\infty.
\]
Actually, the above idea also holds for sequences of anti-tiles of
order greater or equal to $r$. Explicitly, let $r\in\N$ and let
$\{A_{j}\}_{j=0}^{\infty}\subseteq2^{\N_{0}}$ be a sequence of anti-tiles,
and assume, for every $j\in\N_{0}$, that $A_{j}$ has order $r_{j}\geq r$
and has step-length $s_{j}\in r\N$ with $s_{j}\geq r_{j}$. Then,
for any infinite set $K\subseteq\N_{0}$, $T\in\cal{PT}_{r}$ and
$b\in\N$,
\[
\abs{\parenth{\bigcup_{j\in K}A_{j}+\parenth{\sum_{i=0}^{j-1}(2^{r_{i-1}}-1)s_{i}}+1}\backslash\parenth{\bigcup_{j=0}^{\infty}(T+jr)+b}}=\infty.
\]

To establish $\tileideal\cap\reciprocalideal\Subset\banachdensideal\cap\reciprocalideal$
and $\tileideal\Subset\banachdensideal$ in the next lemma, the idea
is to consider an infinite collection disjoint translated anti-tiles
of strictly increasing order. Every union of an infinite subcollection
of these translated anti-tiles cannot be contained in an any infinite
tiling (except for $\N$) which forces this union to have a tile-density
of $1$. Furthermore, by choosing the step lengths of the anti-tiles
in this infinite collection to increase sufficiently quickly, it can
be arranged that every union of an infinite subcollection of these
anti-tiles have zero Banach density and also have convergent reciprocal
series.
\begin{lem}
\label{lem:much-larger-E}The much larger inclusions $\tileideal\cap\reciprocalideal\Subset\banachdensideal\cap\reciprocalideal$
and $\tileideal\Subset\banachdensideal$ hold.
\end{lem}

\begin{proof}
Let $j\in\N$ and let $A_{j}$ be an anti-tile of order $j$ and step-length
$j!$.

Define $B_{j}:=A_{j}+\sum_{i=0}^{j-1}(2^{i}-1)i!+1$. The collection
$\{B_{j}\}_{j\in\N}$ is a collection of non-empty pairwise disjoint
finite subsets of $\N$.

Let $K\subseteq\N$ be any infinite set. We claim that $\bigcup_{k\in K}B_{k}\notin\tileideal$.
By the definition of anti-tiles, there exists no $R\in\cal T$, except
for $R=\N$, so that $\bigcup_{k\in K}B_{k}\subseteq R\mod\finideal$.
Therefore $\tiledensity\parenth{\bigcup_{k\in K}B_{k}}=1$ so that
$\bigcup_{k\in K}B_{k}\notin\tileideal$.

We claim $\bigcup_{k\in K}B_{k}\in\reciprocalideal$. For every $j\in\N$,
\begin{align*}
\sum_{b\in B_{j}}b^{-1} & \leq\sum_{k=1}^{2^{j}-1}\parenth{kj!+\sum_{i=0}^{j-1}(2^{i}-1)i!}^{-1}\\
 & \leq\sum_{k=1}^{2^{j}-1}\parenth{kj!}^{-1}\\
 & \leq(2^{j}-1)/j!\\
 & \leq2^{j}/j!.
\end{align*}
Therefore the reciprocal series of $\bigcup_{k\in K}B_{k}$ converges,
and hence $\bigcup_{k\in K}B_{k}\in\reciprocalideal$.

We claim that $\bigcup_{k\in K}B_{k}\in\banachdensideal$. Let $v\in\N_{0}$
and $q\in\N$. Let $j_{0}\in\N$ be the smallest number such that
$v+1<\sum_{i=0}^{j_{0}-1}(2^{i}-1)i!+1$ and $k_{0}\in\N$ the largest
number such that $\sum_{i=0}^{k_{0}+1}(2^{i}-1)i!+1<v+q$. We assume
that $q$ is large enough so that strict inequalities hold in 
\[
0<v+1<\sum_{i=0}^{j_{0}-1}(2^{i}-1)i!+1<\sum_{i=0}^{k_{0}}(2^{i}-1)i!+1<v+q.
\]
Then
\begin{align*}
\frac{\abs{\bigcup_{k\in K}B_{k}\cap[v+1,v+q]}}{q} & \leq\frac{\abs{A_{j_{0}}}+\abs{A_{j_{0}+2}}+\ldots+\abs{A_{k_{0}}}}{q}\\
 & \leq\frac{\sum_{j=j_{0}}^{k_{0}}(2^{j}-1)}{\sum_{i=0}^{k_{0}}(2^{i}-1)i!-\sum_{i=0}^{j_{0}-1}(2^{i}-1)i!}\\
 & \leq\frac{\sum_{j=j_{0}}^{k_{0}}(2^{j}-1)}{\sum_{i=j_{0}}^{k_{0}}(2^{i}-1)i!}\\
 & \leq\frac{k_{0}(2^{k_{0}}-1)}{\sum_{i=j_{0}}^{k_{0}}(2^{i}-1)i!}\\
 & \leq\frac{k_{0}(2^{k_{0}}-1)}{k_{0}!(2^{k_{0}}-1)}\\
 & =\frac{1}{(k_{0}-1)!}.
\end{align*}
Because $k_{0}\to\infty$ as $q\to\infty$ we obtain $b\parenth{\bigcup_{k\in K}B_{k}}=0$
and hence $\bigcup_{k\in K}B_{k}\in\banachdensideal$.
\end{proof}
\begin{cor}
\label{cor:much-larger-inclusion-relations}The following much larger
inclusions hold:

\begin{center}
\begin{tikzcd}  
	\finideal \arrow[Subseteq]{r} &  
	\tileideal\cap\reciprocalideal \arrow[Subseteq]{r}\arrow[Subseteq]{d} & 
	\banachdensideal\cap\reciprocalideal \arrow[Subseteq]{r}\arrow[Subseteq]{d} & 
	\reciprocalideal\arrow[Subseteq]{d} & \\
	&  
	\tileideal \arrow[Subseteq]{r} & 
	\banachdensideal \arrow[Subseteq]{r} & 
	\densideal \arrow[Subseteq]{r} & 
	2^{\N}.
\end{tikzcd}
\end{center}
\end{cor}

\begin{proof}
This follows from Lemmas~\ref{lem:much-larger-A} through~\ref{lem:much-larger-E}.
\end{proof}
\begin{cor}
\label{cor:examples}The following quotient maps do not admit uniformly
continuous right inverses:

\begin{multicols}{2} 
\begin{enumerate}
\item []$\ellinfty\to\ellinfty/c_{0}.$
\item []$\ellinfty\to\ellinfty/(\tileideal\cap\reciprocalideal)c_{0}.$
\item []$\ellinfty\to\ellinfty/(\banachdensideal\cap\reciprocalideal)c_{0}.$
\item []$\ellinfty\to\ellinfty/\reciprocalideal c_{0}.$
\item []$\ellinfty\to\ellinfty/\tileideal c_{0}.$
\item []$\ellinfty\to\ellinfty/\banachdensideal c_{0}$.
\item []$\ellinfty\to\ellinfty/\densideal c_{0}$.
\end{enumerate}
\end{multicols}

\begin{multicols}{2} 
\begin{enumerate}
\item []$\densideal c_{0}\to\densideal c_{0}/c_{0}.$
\item []$\densideal c_{0}\to\densideal c_{0}/(\tileideal\cap\reciprocalideal)c_{0}.$
\item []$\densideal c_{0}\to\densideal c_{0}/(\banachdensideal\cap\reciprocalideal)c_{0}.$
\item []$\densideal c_{0}\to\densideal c_{0}/\reciprocalideal c_{0}.$
\item []$\densideal c_{0}\to\densideal c_{0}/\tileideal c_{0}.$
\item []$\densideal c_{0}\to\densideal c_{0}/\banachdensideal c_{0}$.
\end{enumerate}
\end{multicols}

\begin{multicols}{2} 
\begin{enumerate}
\item []$\banachdensideal c_{0}\to\banachdensideal c_{0}/c_{0}.$
\item []$\banachdensideal c_{0}\to\banachdensideal c_{0}/(\tileideal\cap\reciprocalideal)c_{0}.$
\item []$\banachdensideal c_{0}\to\banachdensideal c_{0}/(\banachdensideal\cap\reciprocalideal)c_{0}.$
\item []$\banachdensideal c_{0}\to\banachdensideal c_{0}/\tileideal c_{0}.$
\end{enumerate}
\end{multicols}

\begin{multicols}{2} 
\begin{enumerate}
\item []$\reciprocalideal c_{0}\to\reciprocalideal c_{0}/c_{0}.$
\item []$\reciprocalideal c_{0}\to\reciprocalideal c_{0}/(\tileideal\cap\reciprocalideal)c_{0}.$
\item []$\reciprocalideal c_{0}\to\reciprocalideal c_{0}/(\banachdensideal\cap\reciprocalideal)c_{0}.$
\end{enumerate}
\end{multicols}

\begin{multicols}{2} 
\begin{enumerate}
\item []$\tileideal c_{0}\to\tileideal c_{0}/c_{0}.$
\item []$\tileideal c_{0}\to\tileideal c_{0}/(\tileideal\cap\reciprocalideal)c_{0}.$
\end{enumerate}
\end{multicols}

\begin{multicols}{2} 
\begin{enumerate}
\item []$(\banachdensideal\cap\reciprocalideal)c_{0}\to(\banachdensideal\cap\reciprocalideal)c_{0}/c_{0}.$
\item []$(\banachdensideal\cap\reciprocalideal)c_{0}\to(\banachdensideal\cap\reciprocalideal)c_{0}/(\tileideal\cap\reciprocalideal)c_{0}.$
\end{enumerate}
\end{multicols}

\begin{multicols}{2} 
\begin{enumerate}
\item []$(\tileideal\cap\reciprocalideal)c_{0}\to(\tileideal\cap\reciprocalideal)c_{0}/c_{0}.$
\end{enumerate}
\end{multicols}
\end{cor}

\begin{proof}
This follows from the combination of Corollary~\ref{cor:much-larger-inclusion-relations}
and Corollary~\ref{cor:jc0 =00005Cto jc0/ic0 has no uniform cont right inverse}.
\end{proof}

    \bibliographystyle{amsplain}
    \bibliography{bibliography}
\end{document}